\documentclass[letterpaper, 10 pt, conference]{ieeeconf}  % Comment this line out
\IEEEoverridecommandlockouts                              % This command is only
\usepackage{amsmath,amssymb,amsfonts}
\usepackage{algorithmic}
\usepackage{graphicx}

\usepackage{textcomp}

\newtheorem{exmp}{Example}[section]
\newtheorem{theorem}{Theorem}[section]

\newtheorem{proposition}[theorem]{Proposition}
\newtheorem{remark}{Remark}[section]

\newtheorem{assum}{Assumption}[section]

\usepackage[symbol]{footmisc}

\DeclareMathOperator*{\argmin}{arg\,min}

\usepackage[dvipsnames]{xcolor}

\title{\LARGE \bf
Data-driven Rollout for Deterministic Optimal Control
}

\author{Yuchao Li, Karl H. Johansson, Jonas M\aa rtensson, Dimitri P. Bertsekas% <-this % stops a space
\thanks{This work was supported by the Swedish Foundation for Strategic Research, the Swedish Research
Council, and the Knut and Alice Wallenberg Foundation.}% <-this % stops a space
\thanks{Y. Li, K. H. Johansson, and J. M\aa rtensson are with the division of Decision and Control Systems, KTH Royal Institute of Technology, Sweden,
        {\tt\small yuchao,jonas1,kallej@kth.se}}%
\thanks{D. P. Bertsekas is Fulton Professor of Computational Decision Making, ASU, Tempe, AZ, and McAfee Professor of Engineering, MIT, Cambridge, MA,
        {\tt\small dbertsek@asu.edu, dimitrib@mit.edu}}
}

\begin{document}

\maketitle
\thispagestyle{empty}
\pagestyle{empty}

%%%%%%%%%%%%%%%%%%%%%%%%%%%%%%%%%%%%%%%%%%%%%%%%%%%%%%%%%%%%%%%%%%%%%%%%%%%%%%%%
\begin{abstract}
We consider deterministic infinite horizon optimal control problems with nonnegative stage costs. We draw inspiration from learning model predictive control scheme designed for continuous dynamics and iterative tasks, and propose a rollout algorithm that relies on sampled data generated by some base policy. The proposed algorithm is based on value and policy iteration ideas, and applies to deterministic problems with arbitrary state and control spaces, and arbitrary dynamics. It admits extensions to problems with trajectory constraints, and a multiagent structure.
\end{abstract}

%%%%%%%%%%%%%%%%%%%%%%%%%%%%%%%%%%%%%%%%%%%%%%%%%%%%%%%%%%%%%%%%%%%%%%%%%%%%%%%%
\section{INTRODUCTION}\label{sec:introduction}
Many practical challenges lend themselves to be formulated as deterministic infinite horizon optimal control problems. Those problems can be addressed in principle by dynamic programming (DP) methodology via value iteration (VI), and policy iteration (PI) algorithms. However, VI and PI are often limited by large problem size. Instead, one needs to resort to various approximate, sample-based schemes that are based on VI and PI ideas. Among them, rollout has stood out due to its simplicity and reliability. For problems with continuous components, model predictive control (MPC) is another major option. It relies on constrained optimization and has been developed independently of DP. In learning MPC (LMPC) one uses sampled trajectories for obtaining solutions, in a way that is similar to rollout. In this paper we explore these connections further, we develop a rollout algorithm that extends LMPC to arbitrary state and control spaces, system dynamics, and we discuss applications to multiagent problems.
  
We consider problems involving stationary dynamics
\begin{equation}
\label{eq:dynamics}
    x_{k+1}=f(x_k,u_k),\quad k=0,\,1,\,\dots,
\end{equation}
where $x_k$ and $u_k$ are state and control at stage $k$, which belong to state and control spaces $X$ and $U$, respectively, and $f$ maps $X\times U$ to $X$. The control $u_k$ must be chosen from a nonempty constraint set $U(x_k)\subset U$ that may depend on $x_k$. The cost of applying $u_k$ at state $x_k$ is denoted by $g(x_k,u_k)$, and is assumed to be nonnegative:
\begin{equation}
\label{eq:cost}
    0\leq g(x_k,u_k)\leq\infty,\quad x_k\in X,\,u_k\in U(x_k).
\end{equation}
By allowing an infinite value of $g(x,u)$ we can implicitly introduce state constraints: a pair $(x,u)$ is infeasible if $g(x,u)=\infty$. We consider feedback policies of the form $\{\mu,\,\mu,\,\dots\}$, with $\mu$ being a function mapping $X$ to $U$ and satisfying $\mu(x)\in U(x)$ for all $x$. This type of polices are called \emph{stationary}. When no confusion arises, with a slight abuse of notation, we also denote $\{\mu,\,\mu,\,\dots\}$ as $\mu$. There is no restriction on the state space $X$, and the control space $U$, so the problems addressed here are very general.

The \emph{cost function} of a policy $\mu$, denoted by $J_\mu$, maps $X$ to $[0,\infty]$, and is defined at any initial state $x_0 \in X$, as
\begin{equation}
\label{eq:pi_cost_function}
    J_\mu(x_0)= \sum_{k=0}^\infty g(x_k,\mu(x_k)),
\end{equation}
where $x_{k+1}=f(x_k,\mu(x_k))$, $k=0,\,1,\,\dots$. The optimal cost function $J^*$ is defined pointwise as
\begin{equation}
    J^*(x_0)=\inf_{u_k\in U(x_k),\ k=0,1,\ldots\atop x_{k+1}=f(x_k,u_k),\ k=0,1,\ldots}\sum_{k=0}^\infty  g(x_k,u_k).
\end{equation}
A stationary policy $\mu^*$ is called optimal if
\begin{equation*}
    J_{\mu^*}(x)=J^*(x),\quad \forall x\in X.
\end{equation*}

In the context of DP, one hopes to obtain a stationary optimal policy $\mu^*$. Two exact DP algorithms for addressing the problem are VI and PI. However, these algorithms may be intractable, due to the \emph{curse of dimensionality}, which refers to the explosion of the computation as the cardinality of $X$ increases. Thus, various approximation schemes have been proposed, where one aims to obtain some suboptimal policy $\Tilde{\mu}$ such that $J_{\Tilde{\mu}}\approx J^*$. 

One such scheme is MPC, which is widely applied to problems with continuous state and control spaces, stage-decoupled state and control constraints, and continuous dynamics, without requiring continuous state and control space discretization. For a given state constraint set, denoted by $C$, and at any $x \in C$, the MPC algorithm solves the constrained optimization problem
\begin{subequations}
\label{eq:mpc}
    \begin{align}
	\min_{\{u_k\}_{k=0}^{\ell-1}}& \quad \sum_{k=0}^{\ell-1} g(x_k,u_k)+G(x_\ell)\\
	\mathrm{s.\,t.} & \quad x_{k+1}=f(x_k,u_k),\,k = 0,...,\ell-1,\label{eq:mpc_dynamics}\\
	& \quad x_k\in C,\,  u_k \in U(x_k), \ k = 0,...,\ell-1, \label{eq:mpc_inputconstraints}\\
	& \quad  x_0 = x. \label{eq:mpc_intialcondition}
	\end{align}
\end{subequations}
The function $G$ is some real-valued function, which is often designed as an approximation to the cost function of a certain policy, as given in \eqref{eq:pi_cost_function}, while fulfilling additional requirements that we will discuss later. With the optimal attained at $(\Tilde{u}_0,\Tilde{u}_1,\dots,\Tilde{u}_{\ell-1})$, the policy $\Tilde{\mu}$ applied to the system is given by setting $\Tilde{\mu}(x)=\Tilde{u}_0$. The above procedure is repeated to select the control at each stage.

There are two challenges in the MPC scheme introduced above. First, investigating conditions under which the MPC problem \eqref{eq:mpc} is feasible. Second, designing the function $G$ so that the resulting system is stable. Those two challenges are often addressed independently, with independent lines of analysis for feasibility and stability. With abundant capability to generate data, there have been many proposals of MPC variants that provide guarantee of both feasibility and stability while taking advantage of data. A notable approach, relevant to our study, is a scheme known as LMPC introduced in \cite{rosolia2017learning}. It is designed for problems with continuous state and control spaces, continuous dynamics, and problems where the same initial $x_0$ is repeatedly visited. Its analysis follows the classic lines of arguments, where the above challenges are investigated independently.      

When going beyond the scope addressed by MPC, a reliable and simple methodology is \emph{rollout}, which is couched on rigorous VI and PI based analysis. It requires availability of a policy $\mu^0$, referred to as a \emph{base policy}, and at a given state $x_0$, it computes the values of $J_{\mu^0}$ for all subsequent states of interests. At any state, it solves an optimization problem similar to \eqref{eq:mpc}, with $J_{\mu^0}$ in place of $G$ (see Section~\ref{sec:main} for a precise definition). This is referred to as rollout with $\ell$-step lookahead. The policy given by the computation is referred to as the \emph{rollout policy}. When $\ell=1$, rollout is one step of the PI algorithm applied to $x_0$. Rollout and MPC have clear similarities, and it is not surprising that they are closely related. Still, the analysis of rollout revolves around a concept known as \emph{sequential improvement condition}, introduced in \cite{bertsekas1997rollout} and \cite{bertsekas2005dynamic},
\begin{equation}
\label{eq:sequential_imp1}
    \inf_{u\in U(x)} \big\{g(x,u)+J_\mu\big(f(x,u)\big)\big\}\leq J_\mu(x),
\end{equation}
and the constraints considered may be different in nature to the ones considered in MPC, such as constraints imposed on entire trajectories. While rollout is a less ambitious method than the VI and PI algorithms, the rollout policy $\Tilde{\mu}$ is guaranteed to outperform the base policy $\mu^0$. In fact, rollout can be viewed as one step of Newton's method for solving a fixed point equation regarding $J^*$, as explained in \cite{bertsekas2021lessons}. Since Newton's method has a superlinear convergence property, the performance improvement achieved by the policy $\Tilde{\mu}$ is often dramatic and sufficient for practical purposes.

In this paper, we draw inspiration from LMPC and propose a new variant of rollout. It uses sampled trajectories, and requires access to $J_{\mu^0}(x_\ell)$ for only some $x_\ell$ reached from $x_0$ in $\ell$ steps, instead of all such $x_\ell$ as is required by traditional rollout with $\ell$ step lookahead. The states $x$ for which the values $J_{\mu^0}(x)$ are available form a subset $S^0\subset X$. This set is assumed to have the following invariance property under the base policy $\mu^0$:
\begin{equation}
\label{eq:s0_intro}
    f\big(x,\mu^0(x)\big)\in S^0,\quad \forall x\in S^0.
\end{equation}
With such a set $S^0$, we will show that the cost improvement property of the rollout algorithm remains intact. One prime case where such a set $S^0$ can be constructed is for the problem with a nonempty stopping set $X_s\subset X$, which consists of cost-free and forward invariant states in the sense that
\begin{equation*}
    g(x,u)=0,\quad f(x,u)\in X_s,\quad \forall x\in X_s,\,u\in U(x).
\end{equation*}
If we can use the base policy $\mu^0$ to generate a trajectory $\{x_0,\mu^0(x_0),x_1,\mu^0(x_1),\dots,x_j\}$ that starts at some initial state $x_0$ and ends at a stopping state $x_j\in X_s$, the values $J_{\mu^0}(x)$ for all $x=x_0,\dots,x_j$ can be computed.

By introducing the rollout viewpoint, we extend the LMPC scheme to problems with arbitrary state and control spaces, arbitrary system dynamics and nonnegative costs. We show that the method admits extensions to problems where there is a trajectory constraint. When the base policy $\mu^0$ is available for all $x\in S^0$, it allows distributed computation, which is useful for multiagent problems. For yet another extension, if multiple such subsets $S^i$ for multiple different policies $\mu^i$ are available, they can be combined together to form a set $S$ and the performance may be further improved. The contributions of our proposed scheme, are as follows:
\begin{itemize}
    \item[(a)] A rollout algorithm that provides cost improvement over a base policy, for arbitrary state and control spaces, and system dynamics with nonnegative stage costs.
    \item[(b)] An extension to trajectory constraints and the sequential improvement theoretical framework that goes with it.
    \item[(c)] Validation of the methodology using examples ranging from hybrid systems, trajectory constrained problems, multiagent problems, and discrete and combinatorial optimization.
\end{itemize}

The paper is organized as follows. In Section~\ref{sec:background} we provide background and references, which place in context our results in relation to the literature. In Section~\ref{sec:main}, we describe our rollout algorithm and its variants, and we provide analysis. In Section~\ref{sec:exmp}, we demonstrate the validity of our algorithm through examples.

%%%%%%%%%%%%%%%%%%%%%%%%%%%%%%%%%%%%%%%%%%%%%%%%%%%%%%%%%%%%%%%%%%%%%%%%%%%%%%%%
\section{Background}\label{sec:background}
The study of nonnegative cost optimal control with an infinite horizon dates back to the thesis and paper \cite{strauch1966negative}, following the seminal earlier research of \cite{blackwell1967positive,blackwell1965discounted}. Owing to its broad applicability, the followup research has been voluminous and comprehensive, with extensive results provided in the book \cite{bertsekas1978stochastic}; see \cite{bertsekas2012dynamic} for a more accessible discussion.

Among those results, if the minima in the relevant optimization are assumed to be attained, our analysis here relies on only two classical results, which hold well beyond the scope of our study. This is one key contributing factor to the wide applicability of rollout in general and our algorithm in particular. These two results are parts (a) and (b) of the following proposition.
\begin{proposition}\label{prop:classic}
Let the nonnegativity condition \eqref{eq:cost} hold. 
\begin{itemize}
    \item[(a)] For all stationary policies $\mu$ we have
    \begin{equation}
    \label{eq:bellman_mu}
        J_\mu(x)=g\big(x,\mu(x)\big)+J_\mu\big(f\big(x,\mu(x)\big)\big),\,\forall x\in X.
    \end{equation}
    \item[(b)] For all stationary policies $\mu$, if a nonnegative function $\Bar{J}:X\to [0,\infty]$ satisfies
    \begin{equation*}
        g\big(x,\mu(x)\big)+\Bar{J}\big(f\big(x,\mu(x)\big)\big)\leq \Bar{J}(x),\,\forall x\in X,
    \end{equation*}
    then $\Bar{J}$ is an upper bound of $J_\mu$, e.g., $J_\mu(x)\leq \Bar{J}(x)$ for all $x\in X$.
\end{itemize}
\end{proposition}

Prop.~\ref{prop:classic}(a) can be found in \cite[Prop.~4.1.2]{bertsekas2012dynamic}. Prop.~\ref{prop:classic}(b) can be found in \cite[Prop.~4.1.4(a)]{bertsekas2012dynamic}, and is established through Prop.~\ref{prop:classic}(a) and the following monotonicity property: For all nonnegative functions $J$ and $J'$ that map $X$ to $[0,\infty]$, if $J(x)\leq J'(x)$ for all $x\in X$, then for all $x\in X,\,u\in U(x)$, we have
    \begin{equation}
    \label{eq:monotone}
        g(x,u)+J\big(f(x,u)\big)\leq g(x,u)+J'\big(f(x,u)\big).
    \end{equation}

It has long been recognized that the application of the exact forms of VI and PI algorithms are computationally prohibitive. This has motivated suboptimal but less computationally intensive schemes, as noted by Bellman in his classical book \cite{bellman1957dynamic}. In this regard, approximation schemes patterned after the VI and PI algorithms have had great success, and the rollout methodology has stood out thanks to its simple form and reliable performance.

Proposed first in \cite{tesauro1996line} for addressing backgammon, rollout has since been modified and extended for combinatorial optimization \cite{bertsekas1997rollout}, stochastic scheduling \cite{bertsekas1999rollout}, vehicle routing \cite{secomandi2001rollout}, and more recently, Bayesian optimization \cite{lam2016bayesian,yue2020non}. It admits variants that are suitable for problems with multiagent structure (suggested first in \cite[Section~6.1.4]{bertsekas1996neuro} and investigated thoroughly in \cite{bertsekas2021multiagent}), and problems with trajectory constraints \cite{bertsekas2005rollout}. It underlies the success of AlphaZero and related high-profile successes in the context of games, when combined with Monte-Carlo simulation and neural networks \cite{silver2017mastering}, \cite{tesauro1996line}.

Even with all those extensions, the existing forms of rollout do not address the cases where the values $J_{\mu^0}$ are available only in some selected states of $X$, not all of them. Situations of this type arise when obtaining $\mu^0$ is a challenge itself and can only be stored for some selective states, as can be the case for the trajectory constrained problems that we consider in this paper. The interest in this type of settings is justified in view of the abundant availability of data, which could be historical trajectories collected and stored off-line. In this work, we address such problems and show that as long as the values of $J_{\mu^0}$ are available over a subset that has an invariance property like \eqref{eq:s0_intro}, a nearly identical analysis to the existing rollout variants is possible and the cost improvement condition remains valid.

MPC is a major methodology that primarily applies to optimal control problems with state and control consisting mainly of continuous components and stage-decoupled constraints. Developed independently from DP \cite{testud1978model}, MPC was recognized as a moving horizon approximation scheme for solving infinite horizon problems in some early works, e.g., \cite{keerthi1988optimal}. MPC is also closely connected to rollout, as was pointed out in \cite{bertsekas2005dynamic}: it can be viewed as rollout with one step lookahead with the base policy $\mu^0$ involving $(\ell-1)$-step optimization (lookahead steps in our discussion is referred as predicting horizon in MPC literature). The central issues in MPC when applied to deterministic problems are the feasibility of the on-line numerical implementation, and the stability of the resulting closed-loop system. The stability analysis often revolves around a Lyapunov function that is bounded in certain form by a continuous function \cite[Theorem~7.2]{borrelli2017predictive}, \cite[Theorem~B.13]{rawlings2017model} when the dynamics is not continuous, and is also related to the sequential improvement condition \eqref{eq:sequential_imp1}, which is central to the cost improvement property of rollout. The feasibility issues are addressed separately, and revolve around the concept of reachability \cite{bertsekas1971control}. In particular, one goal of such analysis is to establish a property known as \emph{recursive feasibility}. It means that the feasibility of the MPC problem at current state implies its feasibility at the subsequent state, driven by the current MPC control. As opposed to these approaches, we will show that a line of analysis centered around the fixed point equation \eqref{eq:bellman_mu} and performance improvement of PI can serve as an effective substitute for feasibility and stability analysis. For this to hold, we require that the cost function values $J_{\mu^0}(x)$ are known for a subset of states $x$ that has the property \eqref{eq:s0_intro}, but place no continuity-type restrictions.

Our MPC variant that involves final cost and constraints, and sampled trajectories, coincides with the exact form of LMPC \cite{rosolia2017learning} when applied to problems with continuous dynamics and time-decoupled state and control constraints. Still, our work generalizes substantially LMPC. It can deal with problems with trajectory constraints, and discontinuous dynamics (like the ones involving hybrid systems), and admits multiagent variants, as will be discussed shortly.
%%%%%%%%%%%%%%%%%%%%%%%%%%%%%%%%%%%%%%%%%%%%%%%%%%%%%%%%%%%%%%%%%%%%%%%%%%%%%%%%

\section{Main results}\label{sec:main}
In this section, we will provide details on the data-driven rollout methods and their analysis. We will first introduce the basic form of the rollout algorithm that deals with unconstrained-trajectory problems. It requires a base policy $\mu^0$ and its corresponding cost function over a subset of the state space. We will then extend the method to handle a general set constraint on the state-control trajectories. The basic form will also be extended to incorporate multiple base policies. A variant of the basic form will be introduced in the end where the minimization steps are simplified. This variant is well-suited for multiagent applications.

Throughout our discussion of the first three variants, we make the following assumption. For the last variant, we make a slightly different assumption.
\begin{assum}\label{assum:standing}
The cost nonnegativity condition \eqref{eq:cost} holds. Moreover, for all functions $J$ that map $X$ to $[0,\infty]$, the minimum in the following optimization
\begin{equation*}
    \inf_{u\in U(x)} \big\{g(x,u)+J\big(f(x,u)\big)\big\}
\end{equation*}
is attained for all $x\in X$.
\end{assum}

\subsection{Basic form of data-driven rollout}
Consider the optimal control problem with dynamics \eqref{eq:dynamics} and stage cost \eqref{eq:cost}. We assume that we have computed off-line a subset $S^0$ of the state space, and the cost function values $J_\mu^0(x)$ for all $x\in S^0$ of a
stationary policy $\mu^0$. The subset and the stationary policy have the property
\begin{equation}
    \label{eq:s0_main}
    f\big(x,\mu^0(x)\big)\in S^0,\quad \forall x\in S^0.
\end{equation}
We define $\Bar{J}_{S^0}:X\to[0,\infty]$ as
\begin{equation*}
    \Bar{J}_{S^0}(x)= J_{\mu^0}(x)+\delta_{S^0}(x),
\end{equation*}
where $\delta_{S^0}:X\to\{0,\infty\}$ is the indicator function of set $S^0$ defined as $0$ if $x\in S^0$ and $\infty$ otherwise. For a given state $x$, the rollout algorithm solves the following problem
\begin{subequations}
\label{eq:rollout}
    \begin{align}
	\min_{\{u_k\}_{k=0}^{\ell-1}}& \quad \sum_{k=0}^{\ell-1} g(x_k,u_k)+\Bar{J}_{S^0}(x_\ell)\\
	\mathrm{s.\,t.} & \quad x_{k+1}=f(x_k,u_k),\,k = 0,...,\ell-1,\label{eq:rollout_dynamics}\\
	& \quad  u_k \in U(x_k), \ k = 0,...,\ell-1, \label{eq:rollout_inputconstraints}\\
	& \quad  x_0 = x. \label{eq:rollout_intialcondition}
	\end{align}
\end{subequations}
The optimal value of the problem is denoted as $\Tilde{J}_{S^0}(x)$ (with subscript indicating the dependence on the set $S^0$), and the minimizing sequence is denoted as $(\Tilde{u}_0,\Tilde{u}_1,\dots,\Tilde{u}_{\ell-1})$. The rollout policy $\Tilde{\mu}$ is defined by setting $\Tilde{\mu}(x)=\Tilde{u}_0$.  

When the state and control spaces are continuous and under time-decoupled constraints, and the set $S^0$ contains $m$ sample points, the solution of the problem \eqref{eq:rollout} requires solving $m$ continuous optimization problems, with each of the $m$ sample points used as a fixed terminal state.  The infinite stage costs $g(x,u)=\infty$ imply constraints imposed on state and control pairs. In this case, our rollout algorithm coincides with the LMPC given in \cite{rosolia2017learning}, where more discussions on computational issues can be found. On the other hand, when trajectory constraints are present, the set $S^0$ may introduce additional constraints; cf. Example~\ref{eg:lmpc_lqr_constr}.

The preceding rollout algorithm may be viewed as a generalization of the originally proposed MPC algorithm, whereby we first drive the state to $0$ after $\ell$ steps ($\{0\}$ is the set $S^0$ in our algorithm), and then use the policy that keeps the state at $0$ (this is the policy $\mu^0$ in our algorithm). For our rollout policy, we have the following property hold.

\begin{proposition}\label{prop:rollout_basic}
Let Assumption~\ref{assum:standing} hold. For the rollout policy $\Tilde{\mu}$ defined by \eqref{eq:rollout} with set $S^0$ satisfying the invariance property \eqref{eq:s0_main} under the base policy $\mu^0$, we have 
\begin{equation}
\label{eq:rollout_basic}
    J_{\Tilde{\mu}}(x)\leq \Tilde{J}_{S^0}(x)\leq \Bar{J}_{S^0}(x),\quad \forall x\in X.
\end{equation}
\end{proposition}

\begin{proof}
Using a principle of optimality argument (see, e.g., \cite[Section~1.3]{bertsekas2017dynamic}), the optimization problem \eqref{eq:rollout} can be viewed as performing $\ell$ VI iterations with $J_0=\Bar{J}_{S^0}$, generating $\{J_k\}_{k=0}^\ell$ according to 
\begin{equation}
\label{eq:rollout_vi}
    J_{k+1}(x)=\min_{u\in U(x)}\big\{g(x,u)+J_k\big(f(x,u)\big)\big\},
\end{equation}
and with the rollout policy at $x$ given by
\begin{equation}
\label{eq:rollout_vi_policy}
    \Tilde{\mu}(x)\in \argmin_{u\in U(x)}\big\{g(x,u)+J_{\ell-1}\big(f(x,u)\big)\big\},
\end{equation}
with $\Tilde{J}_{S^0}(x)=J_\ell(x)$. With such a view, we will first show that the sequential improvement condition holds, based upon which the finite sequence $\{J_k\}_{k=0}^\ell$ will be shown to be monotonically decreasing. Finally, by applying Prop.~\ref{prop:classic}(b), the monotonicity property \eqref{eq:monotone} and Prop.~\ref{prop:classic}(a) in this order, we can prove the desired inequality \eqref{eq:rollout_basic}.

By regarding rollout as $\ell$-step VI, we first note that for all $x\not\in S^0$, $J_1(x)\leq J_0(x)=\Bar{J}_{S^0}(x)=\infty$. For $x\in S^0$, we have $J_0(x)=\Bar{J}_{S^0}(x)=J_{\mu^0}(x)$ and $\Bar{J}_{S^0}\big(f\big(x,\mu(x)\big)\big)=J_{\mu^0}\big(f\big(x,\mu(x)\big)\big)$, as per \eqref{eq:s0_main}. Then we have that
\begin{align}
    J_1(x)= &\min_{u\in U(x)}\big\{g(x,u)+\Bar{J}_{S^0}\big(f(x,u)\big)\big\}\nonumber\\
    \leq & g\big(x,\mu^0(x)\big)+\Bar{J}_{S^0}\big(f\big(x,\mu^0(x)\big)\big)\nonumber\\
    =&g\big(x,\mu^0(x)\big)+J_{\mu^0}\big(f\big(x,\mu^0(x)\big)\big)=J_0(x),\label{eq:sequential_basic}
\end{align}
where the first equality is per \eqref{eq:rollout_vi}, the second equality is due to $\Bar{J}_{S^0}\big(f\big(x,\mu^0(x)\big)\big)=J_{\mu^0}\big(f\big(x,\mu^0(x)\big)\big)$ for $x\in S^0$, and the last equality is due to that the fixed point equation holds for $\mu^0$, as per Prop.~\ref{prop:classic}(a), and $J_0(x)=J_{\mu^0}(x)$ for $x\in S^0$.

With $J_1(x)\leq J_0(x)$, we apply induction and assume that $J_{k}(x)\leq J_{k-1}(x)$ for all $x$. Then we have
\begin{align*}
    J_{k+1}(x)=&\min_{u\in U(x)}\big\{g(x,u)+J_k\big(f(x,u)\big)\big\}\\
    \leq& \min_{u\in U(x)}\big\{g(x,u)+J_{k-1}\big(f(x,u)\big)\big\}=J_k(x).
\end{align*}
Therefore, we have that $\Tilde{J}_{S^0}(x)=J_\ell(x)\leq J_{\ell-1}(x)\leq  J_0(x)=\Bar{J}_{S^0}(x)$. Since the rollout policy $\Tilde{\mu}(x)$ is defined by \eqref{eq:rollout_vi_policy}, we then have 
\begin{equation}
\label{eq:feasibility}
    g\big(x,\Tilde{\mu}(x)\big)+J_{\ell-1}\big(f\big(x,\Tilde{\mu}(x)\big)\big)=J_\ell(x)\leq J_{\ell-1}(x)
\end{equation}
for all $x\in X$. Applying Prop.~\ref{prop:classic}(b) with $\Bar{J}$ as $J_{\ell-1}$ here, we have $J_{\Tilde{\mu}}(x)\leq J_{\ell-1}(x)$. By the monotonicity property \eqref{eq:monotone}, we have 
\begin{align*}
    J_{\Tilde{\mu}}(x)=&g\big(x,\Tilde{\mu}(x)\big)+J_{\Tilde{\mu}}\big(f\big(x,\Tilde{\mu}(x)\big)\big)\\
    \leq& g\big(x,\Tilde{\mu}(x)\big)+J_{\ell-1}\big(f\big(x,\Tilde{\mu}(x)\big)\big)=J_\ell(x),
\end{align*}
where the first equality is due to Prop.~\ref{prop:classic}(a), the last equality is due to the definition of $\Tilde{\mu}$. The proof is thus complete.
\end{proof}

\begin{remark}
The property \eqref{eq:s0_main} for the set $S^0$ is called forward invariance (or positive invariance) in the MPC literature. However, such a set is easy to obtain by recording a trajectory, as opposed to conventional forward invariant set where major computation may be needed. Still, our analysis applies to these settings, where $S^0$ may be an ellipsoid or polyhedron, provided that the values of $J_{\mu^0}(x)$ are known for all $x\in S^0$. 
\end{remark}

\begin{remark}
The recursive feasibility of the optimization \eqref{eq:rollout} is implied by sequential improvement \eqref{eq:sequential_basic}. To see this, for a given $x$, we denote by $x'$ its subsequent state under rollout, e.g., $x'=f\big(x,\Tilde{\mu}(x)\big)$. The condition $\Tilde{J}_{S^0}(x)=g\big(x,\Tilde{\mu}(x)\big)+J_{\ell-1}(x')<\infty$ implies $J_{\ell-1}(x')<\infty$. Since sequential improvement implies \eqref{eq:feasibility}, we have $\Tilde{J}_{S^0}(x')=J_\ell(x')\leq J_{\ell-1}(x')<\infty$. In other words, due to the construction of our algorithm and nonnegativity of stage costs, the function $\Tilde{J}_{S^0}$ is monotonically decreasing along the trajectory under the rollout policy $\Tilde{\mu}$. 
\end{remark}

\begin{remark}
The stability of the closed-loop system obtained by the rollout algorithm should be addressed independently of the cost improvement property. Note that the cost function of $\Tilde{\mu}$ is bounded by $\Tilde{J}_{S^0}$. Thus, starting with some $x$ such that $\Tilde{J}_{S^0}(x)<\infty$, the resulting trajectory is guaranteed to have a finite cost, which essentially guarantees stability.
\end{remark}

\begin{remark}\label{rmk:robust}
We note that the inequality \eqref{eq:rollout_basic} holds for all $x\in X$, apart from the elements in $S^0$. This means the rollout method has inherited a robustness property in the MPC context, in the sense that if the subsequent state $\Tilde{x}'$ is different from the anticipated next state $x'$, possibly due to external disturbance, one only needs to ensure $\Tilde{J}_{S^0}(\Tilde{x})<\infty$ (initial feasibility) in order to achieve finite trajectory cost.
\end{remark}

An important aspect in our rollout algorithm is the choice of the set $S^0$. It is natural to speculate that with increased size of $S^0$, better performance may be obtained. Our analysis partially confirms this intuition, in the sense that a performance bound, rather than the policy cost function, is improved, as we will show shortly.

Let $S^0$ and $\Bar{S}^0$ denote two sets satisfying the invariance property \eqref{eq:s0_main} with respect to a common base policy $\mu^0$, and $S^0$ is a subset of $\Bar{S}^0$. Let $\Tilde{J}_{S^0}(x)$ and $\Tilde{J}_{\Bar{S}^0}(x)$ denote the optimal values of the problem \eqref{eq:rollout} when $\Bar{J}_{S^0}$ and $\Bar{J}_{\Bar{S}^0}$ are applied as the final costs respectively. Since $S^0\subset \Bar{S}^0$, we have $\delta_{\Bar{S}^0}\leq \delta_{S^0}$, which implies that $\Bar{J}_{\Bar{S}^0}\leq\Bar{J}_{S^0}$. By applying the monotonicity property \eqref{eq:monotone}, one can show that the VI sequence generated by starting with $\Bar{J}_{\Bar{S}^0}$ is always upper bounded by the one starting with $\Bar{J}_{S^0}$, which leads to $\Tilde{J}_{\Bar{S}^0}\leq\Tilde{J}_{S^0}$. We state the result just described as the following proposition.

\begin{proposition}\label{prop:rollout_basic_sample}
Let Assumption~\ref{assum:standing} hold, $S^0$ and $\Bar{S}^0$ be two sets satisfying the invariance property \eqref{eq:s0_main} under the common base policy $\mu^0$, and $S^0\subset \Bar{S}_0$. Denote by $\Tilde{J}_{S^0}(x)$ and $\Tilde{J}_{\Bar{S}^0}(x)$ the optimal values of the problem \eqref{eq:rollout} when $\Bar{J}_{S^0}$ and $\Bar{J}_{\Bar{S}^0}$ are applied as the final costs respectively. Then we have $\Tilde{J}_{\Bar{S}^0}(x)\leq\Tilde{J}_{S^0}(x),\, \forall x\in X$.
\end{proposition}

\subsection{Trajectory constrained rollout}
We now provide a variant of our rollout algorithm that handles the constraint imposed on entire trajectories. In particular, a complete state-control trajectory $\{(x_k,u_k)\}_{k=0}^\infty$ is feasible if it belongs to a set $C_\infty$, e.g., $\{(x_k,u_k)\}_{k=0}^\infty\in C_\infty$. Since our rollout algorithm applies to problems with arbitrary state and control spaces, we can transform the trajectory-constrained problem to an unconstrained one by state augmentation.

To this end, we introduce an information state $e_k$ from some set $E$. Given a partial trajectory $\{(x_i,u_i)\}_{i=0}^k$, its information state can be specified through some suitably defined function $I(\cdot)$ as $e_k=I\big(\{(x_i,u_i)\}_{i=0}^k\big)$. The information state subsumes the information contained in the partial trajectory that is relevant to the trajectory constraint $C_\infty$ in the sense that for every $k$, $\{(x_i,u_i)\}_{i=0}^\infty$ is feasible if and only if
$$\{e_k\}\cup\{(x_i,u_i)\}_{i=k+1}^\infty\in C_\infty.$$
Such an information state always exists. In the extreme case, it could be simply $e_k=\{(x_i,u_i)\}_{i=0}^k$. Accordingly, the dynamics of $e_k$ can be defined through some function $h:E\times X\times U\to E$. Thus, via state augmentation, we obtain a trajectory-unconstrained problem. Its state is $(x_k,e_k)$, while its control is still $u_k$. The constraints imposed on the state $(x_k,e_k)$ is specified implicitly through $C_\infty$. More details on state augmentation and dealing with trajectory constraints can be found in, e.g., \cite[Sec.~1.4]{bertsekas2017dynamic}, \cite[Sec.~3.3]{bertsekas2020rollout}.

Given a feasible sample trajectory $\{(x_k,u_k)\}_{k=0}^\infty$ obtained under some policy $\mu^0$, we can construct the sample set $S^0$ for the corresponding trajectory-unconstrained problem as
\begin{align*}
    S^0=\cup_{k=0}^\infty\big\{(x,e)\,\big|\,x=x_k,\,\{e\}\cup\{(x_i,u_i)\}_{i=k+1}^\infty\in C_\infty\big\}.
\end{align*}
Thus, the algorithm of the basic form applies and its results hold for this transformed problems as well.

\subsection{Rollout with multiple policies}
We now consider the case where there are multiple subsets $S^i$ that have the property \eqref{eq:s0_main} with respect to the corresponding policies $\mu^i$. Let $\mathcal{I}$ denote the collection of indices of these sets. We define $\Bar{J}_S:X\to[0,\infty]$ as
\begin{equation}\label{eq:js}
    \Bar{J}_S(x)=\inf_{i\in \mathcal{I}_x}\Bar{J}_{S^i}(x),
\end{equation}
where $\mathcal{I}_x=\{i\in\mathcal{I}\,|\,x\in S^i\}$, with the understanding that $\inf\emptyset = \infty$, so that $S=\cup_{i\in \mathcal{I}}S^i$. In addition, if $x\in S$, we denote as $i^*_x\in\argmin_{\mathcal{I}_x} J_{\mu^i}(x)$, and we have that
\begin{align}
    &\min_{u\in U(x)}\big\{g(x,u)+\Bar{J}_{S}\big(f(x,u)\big)\big\}\nonumber\\
    \leq & g\big(x,\mu^{i^*_x}(x)\big)+\Bar{J}_{S}\big(f\big(x,\mu^{i^*_x}(x)\big)\big)\nonumber\\
    \leq&g\big(x,\mu^{i^*_x}(x)\big)+J_{\mu^{i^*_x}}\big(f\big(x,\mu^{i^*_x}(x)\big)\big)=\Bar{J}_{S}(x),\label{eq:sequential_multiple}
\end{align}
where the second inequality is due to \eqref{eq:js}. To see this, for $x\in S$ with $i^*_x\in\argmin_{\mathcal{I}_x} J_{\mu^i}(x)$, we have $f\big(x,\mu^{i^*_x}(x)\big)\in S^{i^*_x}$ as is required by the property \eqref{eq:s0_main}. Then by applying the basic form of rollout \eqref{eq:rollout_basic} with $\Bar{J}_{S^0}$ replaced by $\Bar{J}_S$, we have the following result hold. The proof is similar of the previous one by replacing \eqref{eq:sequential_basic} with \eqref{eq:sequential_multiple}, and is thus neglected.
\begin{proposition}\label{prop:rollout_multiple}
Let Assumption~\ref{assum:standing} hold. For the rollout policy $\Tilde{\mu}$ defined by \eqref{eq:rollout} with $\Bar{J}_{S^0}$ replaced by $\Bar{J}_S$ given in \eqref{eq:js}, sets $S^i$, $i\in\mathcal{I}$, satisfying the property \eqref{eq:s0_main} with respect to $\mu^i$, and optimal value of the problem \eqref{eq:rollout} denoted as $\Tilde{J}_{S}$, we have $J_{\Tilde{\mu}}(x)\leq \Tilde{J}_S(x)\leq \Bar{J}_{S}(x),\, \forall x\in X.$
\end{proposition}
\begin{remark}\label{rmk:multiple}
If the set $S$ is replaced by some set $\Bar{S}=\cup_{i\in \mathcal{I}}\Bar{S}^i$, where the sets $\Bar{S}^i$, $i\in\mathcal{I}$, have the invariance property \eqref{eq:s0_main} under $\mu^i$, and $S^i\subset\Bar{S}^i$, a result on performance bound improvement, similar to Prop.~\ref{prop:rollout_basic_sample}, can be established.
\end{remark}

\subsection{Simplified rollout}
For yet another variant, we consider again the same problem as in the basic form case. This time, however, given a set $S^0$ satisfying the property \eqref{eq:s0_main}, we assume that we have constructed nonempty control constraints $\Bar{U}(x)\subset U(x)$ such that $\mu^0(x)\in \Bar{U}(x)$ and the minimum is attained in the relevant minimization. We formalize the scheme by first introducing the following assumption.
\begin{assum}\label{assum:simple}
The cost nonnegativity condition \eqref{eq:cost} holds. The set $S^0$ fulfills the property \eqref{eq:s0_main} under the base policy $\mu^0$. Moreover, for all functions $J$ that map $X$ to $[0,\infty]$, the minimum of the following optimization
\begin{equation*}
    \inf_{u\in \Bar{U}(x)} \big\{g(x,u)+J\big(f(x,u)\big)\big\}
\end{equation*}
is attained for all $x\in X$, where $\Bar{U}(x)\subset U(x)$ are nonempty subsets such that $\mu^0(x)\in \Bar{U}(x)$ for all $x\in S^0$.
\end{assum}

With Assumption~\ref{assum:simple}, we introduce a variant of the basic rollout form, which we call simplified rollout. The simplification is achieved by replacing the constraint $u_k\in U(x_k)$ in \eqref{eq:rollout_inputconstraints} with $u_k\in \Bar{U}(x_k)$, $k = 0,...,\ell-1$, while keeping the other part of the algorithm unchanged. This method is relevant
in a multiagent context, where for example the control $u$ is composed of two parts $u=(v,w)$ that are under the disposal of two agents. At a given state $x$, the base policy $\mu^0$ assigns the values $v^0$ and $w^0$ to those agents. Then the minimization 
\begin{equation*}
    \min_{(v,w)\in U(x)}\big\{g\big(x,(v,w)\big)+J\big(f\big(x,(v,w)\big)\big)\big\}
\end{equation*}
may be replaced by first solving 
\begin{equation*}
    \min_{(v,w)\in U(x)}\big\{g\big(x,(v,w)\big)+J\big(f\big(x,(v,w)\big)\big)\big\},\;\mathrm{s.\,t.}\,w=w^0,
\end{equation*}
with minimum attained at $(\Tilde{v},w^0)$. This step is followed by
\begin{equation*}
    \min_{(v,w)\in U(x)}\big\{g\big(x,(v,w)\big)+J\big(f\big(x,(v,w)\big)\big)\big\},\;\mathrm{s.\,t.}\,v=\Tilde{v}.
\end{equation*}
In doing so, a constraint set $\Bar{U}(x)$ is implicitly constructed, cf. Fig.~\ref{fig:simplified}. This scheme can be extended to the case of multiple agents, and with further modifications, can be made suitable for distributed computation (see \cite{bertsekas2021multiagent}, \cite{bertsekas2020rollout}). We state without proof the performance guarantee that is similar to the previous basic form.
\begin{figure}[ht!]
    \centering
    \includegraphics[width=0.7\linewidth]{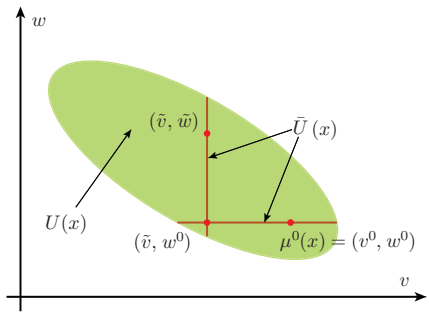}
    \caption{Construction of the set $\Bar{U}(x)$. The original control constraint set $U(x)$ is the green ellipsoid. Starting from the horizontal line segment contained in the set $U(x)$ with $w=w^0$, the first agent optimizes its control and attains the optimum at $(\Tilde{v},w^0)$. The second agent then searches on the line segment contained in the set $U(x)$ with $v=\Tilde{v}$. Thus, the implicitly constructed set $\Bar{U}(x)$ is the union of those two line segments.}
    \label{fig:simplified}
\end{figure}

\begin{proposition}\label{prop:rollout_simple}
Let Assumption~\ref{assum:simple} hold. For the rollout policy $\Tilde{\mu}$ defined by \eqref{eq:rollout} with the constraint $u_k\in U(x_k)$ in \eqref{eq:rollout_inputconstraints} replaced by $u_k\in \Bar{U}(x_k)$ so that $\mu^0(x_k)\in \Bar{U}(x_k)$, $k = 0,...,\ell-1$, we have $J_{\Tilde{\mu}}(x)\leq \Tilde{J}_{S^0}(x)\leq \Bar{J}_{S^0}(x),\, \forall x\in X.$
\end{proposition}
%%%%%%%%%%%%%%%%%%%%%%%%%%%%%%%%%%%%%%%%%%%%%%%%%%%%%%%%%%%%%%%%%%%%%%%%%%%%%%%%

\section{Illustrating examples}\label{sec:exmp}
In this section, we illustrate through examples the validity of our main results. Our first example is from \cite{bemporad1999control} and involves discontinuous dynamics. In the second example, we illustrate a state augmentation technique, as applied to a trajectory-constrained problem modified from \cite[Example~A]{rosolia2017learning}. The third example involves a collision avoidance situation, where simplified rollout is applied. In the last example, we consider a variant of a traveling salesman example, adapted from \cite[Example~1.3.1]{bertsekas2020rollout}. It demonstrates the conceptual equivalence between the rollout with multiple heuristics and LMPC \cite{rosolia2017learning}. The numerical examples~\ref{eg:lmpc_hybrid}, \ref{eg:lmpc_lqr_constr} and \ref{eq:lmpc_multiagent} are solved in MATLAB with toolbox YALMIP \cite{lofberg2004yalmip} and solver \texttt{fmincon}.

With our examples, we will aim to elucidate the following points:
\begin{itemize}
    \item[(1)] The performance improvement of rollout is valid regardless of the nature of the state and control spaces.
    \item[(2)] The proposed algorithm can be used for problems with discontinuous dynamics, consistent with the theory (cf. Example~\ref{eg:lmpc_hybrid}).
    \item[(3)] The rollout algorithms are applicable to trajectory-constrained problems once suitable state augmentation is carried out (cf. Example~\ref{eg:lmpc_lqr_constr}).
    \item[(4)] The simplified rollout enables distributed computation and applies to multiagent problems (cf. Example~\ref{eq:lmpc_multiagent}).
    \item[(5)] Further performance improvement may be obtained via applying multiple heuristics and/or enlarging the sample sets (cf. Example~\ref{eg:lmpc_tsp}).
\end{itemize}

\begin{exmp}\label{eg:lmpc_hybrid}(Discontinuous dynamics) In this example, we apply the basic form of rollout to a problem involving a hybrid system. For this type of problems, conventional analytic approaches for MPC require substantial modifications, while the rollout viewpoint we take applies to such problems without customization. The dynamics is given as 
\begin{equation*}
    f(x,u)=0.8\begin{bmatrix}
        \cos(\beta(x)) & -\sin(\beta(x))\\
        \sin(\beta(x)) & \cos(\beta(x))\end{bmatrix}x+\begin{bmatrix}
        0\\
        1\end{bmatrix}u,
\end{equation*}
where $X=[-10,10]^2\subset \Re^2$, $U= [-1,1]\subset \Re$, where $\Re^n$ denotes the $n$-dimensional Euclidean space. The function $\beta(\cdot)$ is defined as $\frac{\pi}{3}$ if $[1\;0]x \geq 0$ and $-\frac{\pi}{3}$ otherwise. The stage cost is defined as $g(x,u)=x^{\text{T}}x$. The discountinous dynamics is modelled by using binary variables, as in \cite[Example~4.1]{bemporad1999control}. Since the system is stable when the control is set to zero, we use as initial feasible solution $\mu^0(x)=0$ for all $x$. The final stage cost $\Bar{J}_{S^0}$ of rollout is computed accordingly. We compute the trajectory costs of the base policy $\mu^0$, the rollout policy $\Tilde{\mu}$ with $\ell=5$, and a classical MPC policy with horizon $\ell=10$, starting from initial states $(1,1)$ and $(8,-9)$ respectively. The corresponding costs are listed in Table~\ref{table:lmpc_hybrid}. It can be seen the cost improvement property indeed holds. The state trajectories under these policies, starting from initial states $(1,1)$ and $(8,-9)$, are shown in Figs.~\ref{fig:lmpc_hybrid_1_1} and \ref{fig:lmpc_hybrid_8_9} respectively.
\begin{table}[ht!]
\caption{Trajectory costs under different policies in Example~\ref{eg:lmpc_hybrid}.}
\label{table:lmpc_hybrid}
\centering
\begin{tabular}{ |c|c|c|c| } 
\hline
 Initial state & $J_{\mu^0}$ & $J_{\Tilde{\mu}}$ & MPC cost \\
\hline
$[1\, 1]^{\text{T}}$ & $5.5555$ & $5.0162$ & $7.2811$\\
\hline
$[8\, -9]^{\text{T}}$ & $402.7778$ & $318.9486$ & $355.7865$ \\
\hline
\end{tabular}

\end{table}

\begin{figure}[ht!]
    \centering
    \includegraphics[width=0.8\linewidth]{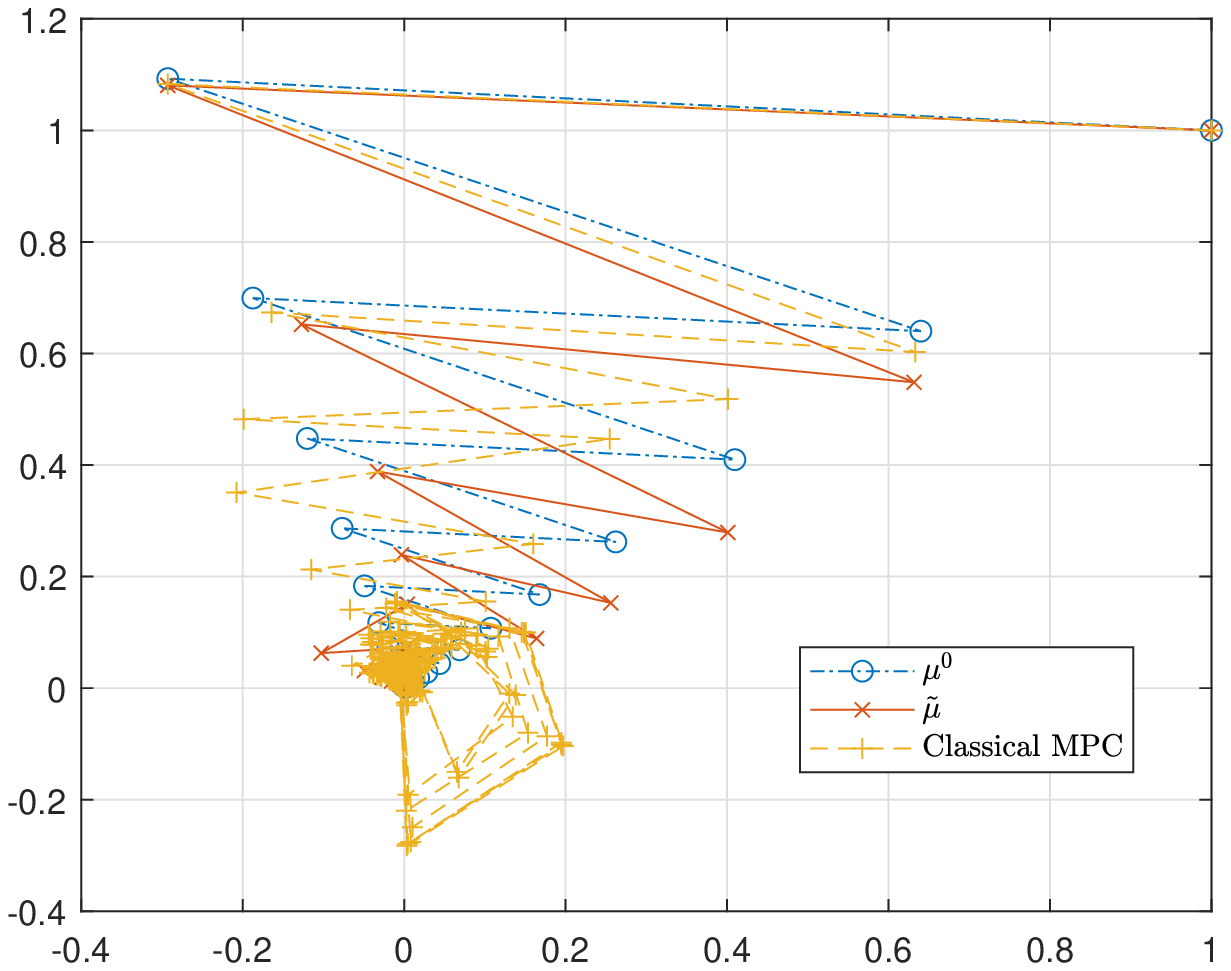}
    \caption{Trajectories under the base policy $\mu^0$, rollout policy $\Tilde{\mu}$ with $\ell=5$, and the classical MPC policy with horizon $10$, starting from $(1,1)$.}
    \label{fig:lmpc_hybrid_1_1}
\end{figure}

\begin{figure}[ht!]
    \centering
    \includegraphics[width=0.8\linewidth]{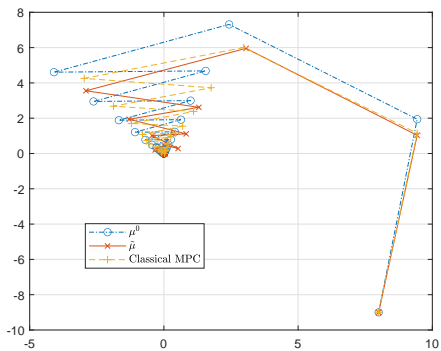}
    \caption{Trajectories under the base policy $\mu^0$, rollout policy $\Tilde{\mu}$ with $\ell=5$, and the classical MPC policy with horizon $10$, starting from $(8,-9)$.}
    \label{fig:lmpc_hybrid_8_9}
\end{figure}

\end{exmp}

\begin{exmp}\label{eg:lmpc_lqr_constr}(Trajectory-constrained problem) We consider a trajectory-constrained problem, and demonstrate the use of state augmentation. The system dynamics is given as $f(x,u)=Ax+Bu$, where 
\begin{equation*}
    A = \begin{bmatrix}
        1 & 1\\
        0 & 1\end{bmatrix},\;B=\begin{bmatrix}
        0\\
        1\end{bmatrix}.
\end{equation*}
The stage cost is given as $g(x,u)=x^{\text{T}}x+u^2+\delta_C(x)$ with constraint set $C=[-4,4]^2$, $U(x)=U=[-1,1]$. We use the initial state $x^0=[-3.95\, -0.05]^{\text{T}}$. The lookahead length is $\ell=4$. All of above settings are the same as those given in \cite[Example~A]{rosolia2017learning}. In addition to the box constraints, we impose constraints on entire trajectories. In particular, we require that a complete trajectory driven under some control sequence $\{u_k\}_{k=0}^\infty$ shall consume energy no more than $0.5$, e.g., $\sum_{k=0}^\infty u_k^2\leq0.5$. This type of constraint goes beyond what is accounted by MPC methodology that is couched on Lyapunov-based analysis. However, the DP framework is flexible and admits modifications that can handle the problem. To this end, we apply state augmentation and define as our new state $y=(x,e)$ where $e$ stands for energy budget that takes values in the interval $[0,0.5]$, and its dynamics is $e_{k+1}=e_k-u_k^2$.

Suppose we have one trajectory $\{x_0,x_1,\dots\}$ under policy $\mu^0$ that fulfills the state, control constraints and trajectory constraint. Then the sample set $S^0$ is given as
\begin{equation*}
    S^0 = \cup_{i=0}^\infty \bigg\{(x,e)\,\bigg|\,x=x_i,\,e+\sum_{j=i}^\infty(\mu^0(x_j))^2\leq 0.5\bigg\}. 
\end{equation*}
Thus, even if the trajectory $\{x_0,x_1,\dots\}$ has a finite length, it `generates' infinite samples in our augmented state space. So from the perspective of the augmented state $y$, the sample set $S^0$ contains infinitely many trajectories and the basic form of rollout can be applied here. 

The obtained solution, in comparison with the initial feasible solution, and the trajectory-unconstrained suboptimal solutions are shown in Fig.~\ref{fig:lmpc_lqr_constrained}. The initial feasible solution has a total cost $74.7492$, with energy consumption $0.1853$, thus feasible; the suboptimal solution without trajectory constraint has a total cost $49.9164$. However, it consumes total energy $1.8616$, and is thus infeasible. The suboptimal solution based upon the reformulated state results in a trajectory with total cost $59.4915$ and energy consumption $0.5$. 

\begin{figure}[ht!]

    \centering
    \includegraphics[width=0.8\linewidth]{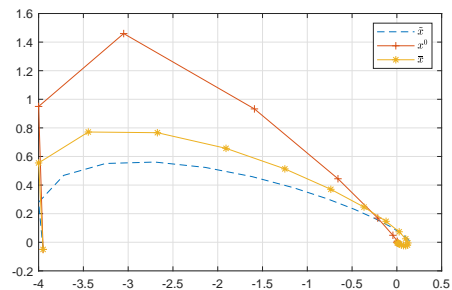}

    \caption{The initial feasible trajectory, marked with dashed blue line and labelled $\Tilde{x}$, has total cost $74.7492$ and energy consumption $0.1853$. The solution given by trajectory-unconstrained problem is in red and labelled with $x^0$. It has total cost $49.9164$, but consumes $1.8616$ energy, thus infeasible. The solution based upon reformulated state is in yellow and it has total cost $59.4915$ and energy consumption $0.5$.}

    \label{fig:lmpc_lqr_constrained}
\end{figure}
\end{exmp}

\begin{exmp}\label{eq:lmpc_multiagent}(Simplified rollout scheme)
To illustrate the simplified rollout when applied to multiagent system, we describe briefly an example of two vehicle path planning. Both vehicles are governed by the kinematic bicycle model discretized with some sampling time $\Delta_t$ and given as
\begin{align*}
    y_{k+1}=&y_k+\Delta_tv_k\cos(\psi_k),\;&&z_{k+1}=z_k+\Delta_tv_k\sin(\psi_k),\\
    v_{k+1}=&v_k+\Delta_ta_k, &&\psi_{k+1}=\psi_k+\Delta_t w_k,
\end{align*}
where $a_k$ and $w_k$ are controls that are subject to box constraints. The states at stage $k$ are denoted collectively by $x_k=(x_k^1,x_k^2)\in \Re^8$, where $x_k^i$, $i=1,2$, is the state associated with vehicle $i$. Similarly, the controls are collectively denoted by $u_k=(u_k^1,u_k^2)\in \Re^4$. The vehicles are tasked to plan shortest paths to their respective targets while remaining safe distances between each other.

At state $x_k$, we assume that a feasible trajectory $S^0=\{x_k,\Bar{x}_{k+1},\Bar{x}_{k+2},\dots\}$ has been obtained under policy $\mu^0=(\mu^0_1,\mu^0_2)$, where $\mu^0_i$, $i=1,2$, prescribes controls to vehicle $i$. When simplified rollout is applied to this problem, vehicle $1$ optimizes over $\{u^1_{k+i}\}_{i=0}^{\ell-1}$ while vehicle $2$ is assumed to apply $\mu_2^0$ (or equivalently, to follow the trajectory $\{\Bar{x}_{k+i}^2\}_{i=1}^\ell$). To be consistent with the rest of the solution, vehicle $1$ also needs to arrive at $\Bar{x}_{k+\ell}^1$ at the end of the trajectory. Once the optimal is attained at $\{\Tilde{u}^1_{k+i}\}_{i=0}^{\ell-1}$, similar procedure is repeated by vehicle $2$, except that vehicle $1$ uses the newly computed $\{\Tilde{u}^1_{k+i}\}_{i=0}^{\ell-1}$. This procedure is repeated with best solution $S^0$ updated accordingly. The obtained solution is guaranteed to outperform initial solution through sequentially optimizing controls of two vehicles.

One important feature of this scheme is that it allows distributed computation, with relevant control signals (or equivalently, planned tentative paths) communicated between each other. A similar scheme in MPC literature is known as cooperative distributed MPC \cite{stewart2011cooperative}. What is novel here is the construction of final state constraint and final cost, which is based upon the collected feasible paths. 

\end{exmp}

\begin{exmp}\label{eg:lmpc_tsp} (A discrete optimization problem) We consider a variant of the classical travelling salesman problem. The explicit illustration of solutions elucidates the mechanism for the performance improvement from the use of multiple heuristics and/or enlarging the sample set. Starting from a city A, the task is to traverse three cities B, C, D, and go back to the city A. All city-wise travels have positive costs, while staying in the same city before completing the trip has an infinite cost. The exact values of costs are immaterial to our purpose, and one can find the corresponding numerical example in \cite[Example~1.3.1]{bertsekas2020rollout}. All trips that end with city A and include all four cities belong to a stopping set $X_s$, e.g., ABCDA belongs to $X_s$ while ABABA does not. Since cities can be repeatedly visited (e.g., ABABA), some policies may have infinite costs. Still, our problem is equivalent to the classical problem in the sense that they admit the same optimal solution.

Suppose that the optimal schedule is ABDCA. Let $\mu^0$ and $\mu^1$ be two policies that bring states to $X_s$, where starting from $\text{A}$, $\mu^0$ results in the complete trajectory $S^0=\{\text{A,\,AC,\,ACD,\,ACDB,\,ACDBA}\}$, and $\mu^1$ follows alphabetic order and gives $S^1=\{\text{A,\,AB,\,ABC,\,ABCD,\,ABCDA}\}$. The policy $\mu^0$ is inferior to $\mu^1$, and neither is optimal. When the basic form with lookahead steps $\ell=2$ and $S^0$ is used, one can see that $\Tilde{\mu}$ coincides with $\mu^0$ as any other choice would have an infinite cost. If rollout with multiple heuristics is applied, with the set $S^0$ replaced by $S=S^0\cup S^1$, the rollout policy would recover the policy $\mu^1$, as shown in Fig.~\ref{fig:lmpc_tsp}. The performance improvement is due to the use of multiple heuristics. In either case, the rollout solutions do not contain cycles like ABABA. In addition, this example shows that when the data-driven rollout converges, it may converge at a suboptimal schedule, namely ABCDA.

\begin{figure}[ht!]
    \centering
    \includegraphics[width=0.6\linewidth]{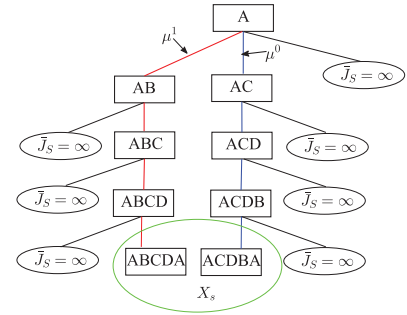}

    \caption{The travelling salesman example with data-driven rollout method. The sampled trajectory under $\mu^0$ is highlighted in blue, and the one under $\mu^1$ is in red. The ellipsoids with black boundaries are subsets of $X$ where $\Bar{J}_S$ is infinity. The ellipsoid with green boundary belongs to $X_s$.}

    \label{fig:lmpc_tsp}
\end{figure}

On the other hand, if a trajectory starting from ABD and under policy $\mu^0$ is recorded and added to $S$, then our method obtains the optimal schedule. The new path that is found by the rollout method is highlighted in orange in Fig.~\ref{fig:lmpc_tsp_explore}. It can be seen that this path can not be generated by either $\mu^0$ or $\mu^1$. This indicates the value of exploration, and is consistent with the Remark~\ref{rmk:multiple}. 

\begin{figure}[ht!]
    \centering
    \includegraphics[width=0.6\linewidth]{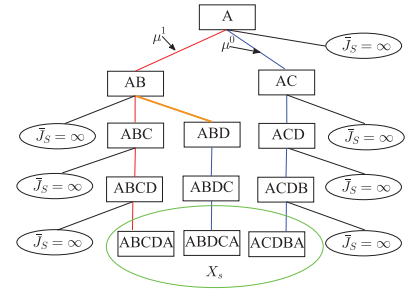}
    \caption{The travelling salesman example with rollout method where new samples ABD, ABDC, ABDCA is added to the sample set $S$. The path highlighted in orange is a new path that is found by rollout, which could not be generated by either $\mu^0$ or $\mu^1$.}
    \label{fig:lmpc_tsp_explore}
\end{figure}
\end{exmp}

%%%%%%%%%%%%%%%%%%%%%%%%%%%%%%%%%%%%%%%%%%%%%%%%%%%%%%%%%%%%%%%%%%%%%%%%%%%%%%%%

\section{CONCLUSIONS}\label{sec:con}
We considered infinite horizon deterministic optimal control problems with nonnegative stage costs. By drawing inspiration from LMPC, we proposed a rollout algorithm variant that applies to problems with arbitrary state and control spaces, discontinuous dynamics, admits extensions for trajectory constrained problems, and applies to multiagent problems. We illustrated the validity of our assertions with various examples from combinatorial optimization, trajectory constrained problems, hybrid system optimization, and multiagent problems. Our method relies upon knowing a perfect model. The robustness issues that arise in the face of model mismatches require further research and numerical experimentation.

%%%%%%%%%%%%%%%%%%%%%%%%%%%%%%%%%%%%%%%%%%%%%%%%%%%%%%%%%%%%%%%%%%%%%%%%%%%%%%%%

\bibliographystyle{IEEEtran}
\bibliography{ref}

\end{document}